\documentclass[12pt]{article}
\usepackage{amsmath,amsfonts,amssymb, amsthm}
\usepackage{hyperref} 
\hypersetup{colorlinks=true, citecolor=blue}
\usepackage{multirow} 
\usepackage{caption} 
\usepackage{kotex}
\usepackage{xcolor} 
\usepackage[sort]{natbib}
\usepackage{fullpage} 
\usepackage{float} 
\usepackage{chngpage} 
\usepackage{lscape}
\usepackage{mathtools}
\usepackage[blocks]{authblk}
\usepackage{booktabs} 
\usepackage{pbox} 
\usepackage{algpseudocode, algorithm}
\usepackage{enumerate}

\newcounter{algsubstate}


\bibpunct{(}{)}{;}{a}{}{,} 

\newtheorem{theorem}{Theorem}
\newtheorem*{theorem*}{Theorem}
\newtheorem{assumption}{Assumption}

\newtheorem{lemma}{Lemma}

\newcommand{\bA}{\boldsymbol{A}}
\newcommand{\bB}{\boldsymbol{B}}

\newcommand{\bD}{\boldsymbol{D}}

\newcommand{\bI}{\boldsymbol{\rm I}}

\newcommand{\bS}{\boldsymbol{S}}

\newcommand{\bU}{\boldsymbol{U}}

\newcommand{\bX}{\boldsymbol{X}}
\newcommand{\bY}{\boldsymbol{Y}}
\newcommand{\bZ}{\boldsymbol{Z}}

\newcommand{\bu}{\boldsymbol{u}}

\newcommand{\bx}{\boldsymbol{x}}

\newcommand{\bzero}{\boldsymbol{0}}
\newcommand{\bone}{\boldsymbol{1}}

\newcommand{\balpha}{\boldsymbol{\alpha}}

\newcommand{\bdelta}{\boldsymbol{\delta}}

\newcommand{\bgamma}{\boldsymbol{\gamma}}
\newcommand{\bmu}{\boldsymbol{\mu}}
\newcommand{\bpi}{\boldsymbol{\pi}}

\newcommand{\bSigma}{\boldsymbol{\Sigma}}

\newcommand{\diag} {{\rm {diag}}}
\newcommand{\tr} {{\rm {tr}}}

\title{Sparse Hanson-Wright Inequality for a Bilinear Form of Sub-Gaussian Variables}

\author[1]{Seongoh Park}

\author[2]{Xinlei Wang}

\author[3]{Johan Lim\footnote{To whom all correspondence should be addressed. Email: \texttt{johanlim@snu.ac.kr}}}

\affil[1]{School of Mathematics, Statistics and Data Science, Sungshin Women's University, Seoul, Korea} 
\affil[2]{Department of Statistical Science, Southern Methodist University, Dallas, TX, USA}
\affil[3]{Department of Statistics, Seoul National University, Seoul, Korea} 

\date{}

\begin{document}
\maketitle

\begin{abstract} 
	\noindent 
In this paper, we derive a new version of Hanson-Wright inequality for a sparse bilinear form of sub-Gaussian variables. Our results are generalization of previous deviation inequalities that consider either sparse quadratic forms or dense bilinear forms. We apply the new concentration inequality to testing the cross-covariance matrix when data are subject to missing. Using our results, we can find a threshold value of correlations that controls the family-wise error rate. Furthermore, we discuss the multiplicative measurement error case for the bilinear form with a boundedness condition.
	
	\vskip0.5cm 
	\noindent {\bf Keywords:} 
	Concentration inequality; covariance matrix; missing data; measurement error; general missing dependency.
\end{abstract} 
\baselineskip 18pt

\section{Introduction}

Let $(Z_{1j}, Z_{2j})$, $j=1,\ldots, n$, be pairs of (possibly dependent) sub-Gaussian random variables with mean zero. Suppose $\gamma_{ij}$ is a random variable that corresponds to $Z_{ij}$ and is independent with $Z_{ij}$. We allow (general) dependency between $\gamma_{1j}$ and $\gamma_{2j}$. We write the random variables by vectors $\bZ_i = (Z_{i1}, \ldots, Z_{in})^\top$, $\bgamma_i = (\gamma_{i1}, \ldots, \gamma_{in})^\top$, $i=1,2$. For a given $n\times n$ non-random matrix $\bA$, \cite{Rudelson:2013} dealt with a concentration inequality of a quadratic form $\bZ_1^\top \bA \bZ_1$, which is known as Hanson-Wright inequality. Recently, this result has been extended in two directions.
On the one hand, \cite{Zhou:2019} included Bernoulli random variables $\bgamma_1$ in the quadratic form, i.e., $(\bZ_1 * \bgamma_1)^\top \bA (\bZ_1 * \bgamma_1)$ where $*$ is the Hadamard (element-wise) product of two vectors (or matrices). On the other hand, \cite{Park:2021} proved that the sub-Gaussian vectors could be distinct $\bZ_1 \neq \bZ_2$  and derived the deviation inequality of a bilinear form $\bZ_1^\top \bA \bZ_2$.

In this paper, considering the two extensions all together, we aim to analyze the concentration of the bilinear form $(\bZ_1 * \bgamma_1)^\top \bA (\bZ_2 * \bgamma_2)$ where $\bgamma_1, \bgamma_2$ are vectors of Bernoulli random variables that may be dependent on each other. Furthermore, we generalize the Bernoulli variables to bounded random variables, which is the first attempt in the literature. Table \ref{tab:prev_result} compares different types of bilinear forms. 
One may consider converting the bilinear form into a quadratic form by concatenating $\bZ_1, \bZ_2$, i.e. 
$$
(\bZ_1 * \bgamma_1)^\top \bA (\bZ_2 * \bgamma_2)= (\bar{\bZ} * \bar{\bgamma})^\top 
\begin{bmatrix}
\bzero & \bA/2\\
\bA/2 & \bzero
\end{bmatrix}(\bar{\bZ} * \bar{\bgamma})
$$
where $\bar{\bZ}=(\bZ_1^\top, \bZ_2^\top)^\top, \bar{\bgamma}=(\bgamma_1^\top, \bgamma_2^\top)^\top$. However, we cannot directly apply the previous results (e.g. Theorem 1 in \cite{Zhou:2019}) because entries in $\bar{\bZ}$ are not independent.

We apply the new version of Hanson-Wright inequality to estimating the high-dimensional cross-covariance matrix when data are subject to either missingness or measurement errors. We treat the problem as multiple testing of individual elements of the matrix and propose an estimator thresholding their sample estimates. To determine the cutoff value, we make use of the extended Hanson-Wright inequality and thus can control the family-wise error rate at a desirable level.

\begin{table}[t]
	\centering
	\begin{tabular}{|c|p{6cm}|p{6cm}|}
		\hline
		& Identical ($\bZ_1=\bZ_2$) & Distinct ($\bZ_1\neq \bZ_2$)\\
		\hline
		Dense & $\bZ_1^\top \bA \bZ_1$, \newline \citep{Rudelson:2013} & $\bZ_1^\top \bA \bZ_2$, \newline \citep{Park:2021} \\
		\hline
		Sparse & $(\bZ_1 * \bgamma_1)^\top \bA (\bZ_1 * \bgamma_1)$, \newline \citep{Zhou:2019} & $(\bZ_1 * \bgamma_1)^\top \bA (\bZ_2 * \bgamma_2)$ \\
		\hline
	\end{tabular}
	\caption{Comparison of bilinear forms in the literature and their references.}
	\label{tab:prev_result}
\end{table}

The paper is organized as follows. In Section 2, we present the main result of $(\bZ_1 * \bgamma_1)^\top \bA (\bZ_2 * \bgamma_2)$, Theorem \ref{thm:HWineq_claim}, and apply it to the problem of testing the cross-covariance matrix in Section 3. In Section 4, we conclude this paper with a brief discussion.

\section{Main result}

We first define the sub-Gaussian (or $\psi_2$-) norm of a random variable $X$ in $\mathbb{R}$ by
$$
||X||_{\psi_2} = \sup_{p \ge 1} \dfrac{(\mathbb{E}|X|^p)^{1/p}}{\sqrt{p}},
$$
and $X$ is called sub-Gaussian if its $\psi_2$-norm is bounded. For a matrix $\bA$, its operator norm $||\bA||_2$ is defined by the square-root of the largest eigenvalue of $\bA^\top \bA$. $||\bA||_F=\sqrt{\sum_{i,j} a_{ij}^2}$. For a vector $\bx$, $||\bx||_2$ is a 2-norm and, $\bD(\bx)$ is a diagonal matrix having its diagonal entries by $\bx$.

We now describe the main theorem of this paper.
\begin{theorem}\label{thm:HWineq_claim}
	Let $(Z_{1j}, Z_{2j})$, $j=1,\ldots, n$, be independent pairs of (possibly correlated) random variables satisfying $\mathbb{E} Z_{1j} =\mathbb{E} Z_{2j}=0$, and
	$||Z_{ij}||_{\psi_2} \le K_i$, $i=1,2$ and $\forall j$. Also, suppose $(\gamma_{1j}, \gamma_{2j})$, $j=1,\ldots, n$, be independent pairs of (possibly correlated) Bernoulli random variables such that $\mathbb{E}\gamma_{ij} = \pi_{ij}$ and $\mathbb{E}\gamma_{1j} \gamma_{2j} = \pi_{12,j}$ for $j=1, \ldots, n$, $i=1,2$.
	Assume $Z_{ij}$ and $\gamma_{i'j'}$ are independent for distinct pairs $(i,j)\neq (i', j')$. Then, we have that for every $t > 0$
	$$
	\begin{array}{l}
	{\rm P} \bigg[
	\big| (\bZ_1 * \bgamma_1)^\top \bA (\bZ_2 * \bgamma_2) - \mathbb{E}(\bZ_1 * \bgamma_1)^\top \bA (\bZ_2 * \bgamma_2)\big| > t 
	\bigg] \\
	\qquad\qquad\qquad \le 2\exp \left\{
	- c \min \left(\dfrac{t^2}{K_1^2 K_2^2 ||\bA||_{F, \pi}^2 }, \dfrac{t}{K_1 K_2 ||\bA||_2}\right)
	\right\},
	\end{array}
	$$
	for some numerical constant $c>0$. For a matrix $\bA=(a_{ij}, 1\le i,j \le n)$, we define $||\bA||_{F, \pi}  = \sqrt{\sum_{j=1}^n \pi_{12,j} a_{jj}^2 + \sum_{i \neq j}a_{ij}^2 \pi_{1i}\pi_{2j}}$.
\end{theorem}
\noindent
Note that Theorem \ref{thm:HWineq_claim} is a combination of the results for diagonal and off-diagonal cases given below.
\begin{lemma}\label{lem:HWineq_diag}
	Assume $\bZ_1, \bZ_2, \bgamma_1, \bgamma_2$ are defined as in Theorem \ref{thm:HWineq_claim}. Let $\bA_{\rm diag}=\diag(a_{11}, \ldots, a_{nn})$ be a diagonal matrix. We denote $\mathbb{E} \gamma_{1j} \gamma_{2j} = \pi_{12,j}$. Then, for $t>0$, we have
	$$
	\begin{array}{l}
	{\rm P} \bigg[
	\big| (\bZ_1 * \bgamma_1)^\top \bA_{\rm diag} (\bZ_2 * \bgamma_2) - \mathbb{E}(\bZ_1 * \bgamma_1)^\top \bA_{\rm diag} (\bZ_2 * \bgamma_2)\big| > t 
	\bigg] \\
	\qquad\qquad\qquad \le 2\exp \left\{
	- c 
	\min \left(\dfrac{t^2}{K_1^2 K_2^2
		\sum_{j=1}^n \pi_{12,j} a_{jj}^2}, \dfrac{t}{K_1 K_2 \max\limits_{1\le j \le n} |a_{jj}|}\right)
	\right\},
	\end{array}
	$$
	for some numerical constant $c>0$.
\end{lemma}
\begin{lemma}\label{lem:HWineq_offdiag}
	Assume $\bZ_1, \bZ_2, \bgamma_1, \bgamma_2$ are defined as in Theorem \ref{thm:HWineq_claim}. Let $\bA_{\rm off}$ be a $n\times n$ matrix with its diagonals zero.  We denote $\mathbb{E} \gamma_{ij} = \pi_{ij}$ for $i=1,2$ and $j=1,\ldots, n$. Then, for $t>0$, we have
	$$
	\begin{array}{l}
	{\rm P} \bigg[
	\big| 
	(\bZ_1 * \bgamma_1)^\top \bA_{\rm off} (\bZ_2 * \bgamma_2) \big| > t 
	\bigg] \\
	\qquad\qquad\qquad  \le 2\exp \left\{
	- c \min
	\left(\dfrac{t^2}{K_1^2 K_2^2 
		\sum_{i \neq j}a_{ij}^2 \pi_{1i}\pi_{2j}
	}, \dfrac{t}{K_1 K_2 || \bA_{\rm off}||_2}\right)
	\right\},
	\end{array}
	$$
	for some numerical constant $c>0$.
\end{lemma}
\noindent
A complete proof of the two lemmas, in spirit of \cite{Zhou:2019}, is provided in Appendix. It is noted that our theorem above can yield Theorem 1.1 in \cite{Zhou:2019} and Lemma 5 in \cite{Park:2021} under the same setting of each.

To handle more general case where we do not have information about mean, we provide the concentration inequality for non-centered sub-Gaussian variables.
\begin{theorem}\label{thm:HWineq_claim_nzm}
	Let $(\tilde{Z}_{1j}, \tilde{Z}_{2j})$, $j=1,\ldots, n$, be independent pairs of (possibly correlated) random variables with $\mathbb{E} \tilde{Z}_{1i}=\mu_{1i}$, $\mathbb{E} \tilde{Z}_{2j}=\mu_{2j}$, and
	$||\tilde{Z}_{ij} - \mu_{ij} ||_{\psi_2} \le K_i$, $i=1,2$ and $\forall j$. Also, suppose $(\gamma_{1j}, \gamma_{2j})$, $j=1,\ldots, n$, be independent pairs of (possibly correlated) Bernoulli random variables such that $\mathbb{E}\gamma_{ij} = \pi_{ij}$ and $\mathbb{E}\gamma_{1j} \gamma_{2j} = \pi_{12,j}$ for $j=1, \ldots, n$, $i=1,2$.
	Assume $\tilde{Z}_{ij}$ and $\gamma_{i'j'}$ are independent for distinct pairs $(i,j)\neq (i', j')$. Then, we have that for every $t > 0$
	$$
	\begin{array}{l}
	{\rm P} \bigg[
	\big| (\tilde{\bZ}_1 * \bgamma_1)^\top \bA (\tilde{\bZ}_2 * \bgamma_2) - \mathbb{E}(\tilde{\bZ}_1 * \bgamma_1)^\top \bA (\tilde{\bZ}_2 * \bgamma_2)\big| > t 
	\bigg] \le d\exp \left\{
	- c \min \left(\dfrac{t^2}{E_1}, \dfrac{t}{E_2}\right)
	\right\},
	\end{array}
	$$
	for some numerical constants $c,d>0$. $E_1$ and $E_2$ are defined by
	$$
	\begin{array}{lll}
	E_1 = \max\Bigg\{ & K_1^2 K_2^2 ||\bA||_{F, \pi}^2, & V_{1,\pi}^2 V_{2,\pi}^2||\bD(\bmu_1)\bA \bD(\bmu_2)||_{F}^2, \\
	& V_{1,\pi}^2 K_2^2 ||\bD(\bmu_1 * \bmu_1)^{1/2}\bA \bD(\bpi_2)^{1/2}||_{F}^2, &
	V_{2,\pi}^2 K_1^2 ||\bD(\bpi_1)^{1/2}\bA \bD(\bmu_2 * \bmu_2)^{1/2}||_{F}^2,\\
	& K_2^2 ||\bA^\top (\bmu_1 * \bpi_1)||_2^2, &
	K_1^2 ||\bA(\bmu_2 * \bpi_2)||_2^2,\\
	& V_{1,\pi}^2 ||\{\bA(\bmu_2 * \bpi_2)\} * \bmu_1||_{2}^2, &
	V_{2,\pi}^2 ||\{\bA(\bmu_1 * \bpi_1)\} * \bmu_2||_{2}^2\Bigg\},\\[1em]
	E_2 = \max\Bigg\{ & K_1 K_2 ||\bA||_2,  &
	V_{1,\pi} V_{2,\pi}||\bD(\bmu_1)\bA \bD(\bmu_2)||_2,\\
	& V_{1,\pi} K_2 ||\bD(\bmu_1)\bA ||_2, &
	V_{2,\pi} K_1 ||\bA \bD(\bmu_2) ||_2 \Bigg\}.
	\end{array}
	$$
	For a matrix $\bA=(a_{ij}, 1\le i,j \le n)$, we define $||\bA||_{F, \pi}  = \sqrt{\sum_{j=1}^n \pi_{12,j} a_{jj}^2 + \sum_{i \neq j}a_{ij}^2 \pi_{1i}\pi_{2j}}$.
	Also, we define $V_{1,\pi}=\max\limits_{1\le i \le n} \pi_{1i}(1-\pi_{1i})$ and $V_{2,\pi}$ similarly.
\end{theorem}
\noindent
The bilinear form of $\tilde{Z}_{ij}$ in Theorem \ref{thm:HWineq_claim_nzm} can be decomposed into the bilinear forms and the linear combinations of centered random variables. Then, we can apply either Theorem \ref{thm:HWineq_claim} or Hoeffding's inequality (Theorem \ref{thm:hoeffing} in Appendix) to them. The detail of the proof can be found in Appendix \ref{app:proof_HWineq_nzm}.

%
%
%

\section{Application to testing the cross-covariance matrix}

A cross-covariance matrix $\bSigma_{XY}$ between two random vectors $\bX=(X_1, \ldots, X_p)^\top$ and $\bY=(Y_1, \ldots, Y_q)^\top$, with its $(k,\ell)$-th entry being $\sigma_{k\ell}^{XY}=\text{Cov}(X_k, Y_\ell)$, refers to the off-diagonal block matrix of the covariance matrix of $\bZ=\big({\bX}^{\top}, {\bY}^{\top} \big)^{\top}$. It is often considered a less important part in the covariance matrix of $\bZ$, $\bSigma_{ZZ}$, and tends to be overpenalized by shrinkage estimators favoring an invertible estimate. 
However, it is a crucial statistical summary in some applications. For example, 
the study of multi-omics data, which aims to explain molecular variations 
at different molecular levels, receives much attention due to 
the public availability of biological big data and the covariation between two different data sources is just as important as that within each data source. 
The main question here is to find pairs (or positions in the matrix) of $X_k$'s and $Y_\ell$'s that present a large degree of association, which can be treated by hypothesis testing:
\begin{equation}\label{eq:global_hypo}
	\mathcal{H}_{0, k\ell}: \sigma_{k\ell}^{XY}=0 \quad \text{vs.} \quad 
	\mathcal{H}_{1, k\ell}: \text{not } \mathcal{H}_{0, k\ell},
\end{equation}
for $1\le k \le p, 1 \le \ell \le q$.

Testing the cross-covariance matrix has not been much explored in literature. 
\cite{Cai:2019} directly address the problem of estimating the cross-covariance matrix. However, they vaguely assume $q \ll p$ and consider simultaneous testing of hypotheses
$$
\mathcal{H}_{0k}: \sigma_{k1}^{XY}=\ldots=\sigma_{kq}^{XY}=0, \quad \text{vs.} \quad 
\mathcal{H}_{1k}: \text{not } \mathcal{H}_{0k}
$$
for $k=1, \ldots, p$. They build Hotelling's $T^2$ statistics for individual hypotheses and decide which rows of $\bSigma_{XY}$ are not 0. Hence, the sparsity pattern in \cite{Cai:2019} is not the same as considered in this paper. Moreover, their method cannot address missing data.

Considering (\ref{eq:global_hypo}) is equivalent to $\mathcal{H}_{0, k\ell}: \rho_{k\ell}^{XY}=0$ where $ \rho_{k\ell}^{XY}=\text{Cor}(X_k, Y_\ell)$, one can also analyze the sample correlation coefficient $\gamma_{k\ell}$. \cite{Bailey:2019} analyzed a universal thresholding estimator of $\gamma_{k\ell}$ based on its asymptotic normality. Though they are interested in estimation of a large correlation matrix, not a cross-correlation matrix directly, their method can be applied to estimation of the cross-covariance matrix. For example, the proposed estimator would be a $p\times q$ matrix with its component $\gamma_{k\ell} 1_{(|\gamma_{k\ell}| > c(n,p,q))}$, and they aim to find a cutoff value $c(n,p,q)$ to control the family-wise error rate (FWER).
However, again, if data are subject to missing, their method is no longer valid.

Here, we address the multiple testing problem for the cross-covariance matrix when some of the data are missing or measured with errors. We apply the modified Hanson-Wright inequalities (Theorem \ref{thm:HWineq_claim}, \ref{thm:HWineq_claim_nzm}, \ref{thm:HWineq_msr_error}, \ref{thm:HWineq_msr_error_nzm}) in order to choose a threshold value that controls FWER. More specifically, we derive concentration results for an appropriate cross-covariance estimator $\hat{\sigma}_{k\ell}^{XY}$ in the following form: 
$$
{\rm P} \left[
\max\limits_{k, \ell} \big| \hat{\sigma}_{k\ell}^{XY} - \sigma_{k\ell}^{XY}\big| > c(n,p,q)
\, \right] \le \alpha, \quad 0 < \alpha <1
$$
under some regularity conditions. We begin with the simplest case where data are completely observed and walk through more complicated cases later.

\subsection{Complete data case}\label{sec:complete_case}
%
We begin with the complete data case. 

\begin{assumption}\label{assump:X_and_Y}
	Assume each components of random vectors $\bX\in\mathbb{R}^p$ and $\bY\in\mathbb{R}^q$ are sub-Gaussian, i.e., it holds for some $K_X, K_Y>0$ 
	\begin{equation}\label{eq:subgaussian}
		\max_{1 \le k \le p} ||X_k - \mathbb{E}X_k||_{\psi_2} \le K_X, \quad
		\max_{1 \le \ell \le q}||Y_\ell - \mathbb{E}Y_\ell||_{\psi_2} \le K_Y.
	\end{equation}
	Let us denote the mean vector and cross-covariance matrix of $\bX$ and $\bY$ as follows:
	\begin{equation}\label{eq:XY_moment}
		\begin{array}{l}
			\mathbb{E}\bX = \bmu^X = (\mu_k^X, 1\le k \le p)^\top, \\
			\mathbb{E}\bY = \bmu^Y  = (\mu_\ell^Y, 1\le \ell \le q)^\top,\\
			\text{Cov}(\bX, \bY) = \bSigma_{XY} = (\sigma_{k\ell}^{XY}, 1\le k \le p, 1 \le \ell \le q).
		\end{array}
	\end{equation}
\end{assumption}

Suppose we observe $n$ independent samples $\{(\bX_i, \bY_i)\}_{i=1}^n$ of $(\bX, \bY)$ in Assumption \ref{assump:X_and_Y}.
Then, the cross-covariance $\sigma_{k\ell}^{XY}$ can be estimated by 
$$
\begin{array}{rcl}
s_{k\ell} &=& \dfrac{1}{n-1}\sum\limits_{i=1}^n (X_{ik} - \bar{X}_k)(Y_{i\ell} - \bar{Y}_\ell).
\end{array}
$$
where $\bar{X}_k$, $\bar{Y}_\ell$ are the sample means corresponding to $\mu_k^X$, $\mu_\ell^Y$.
We can analyze the concentration of each component of $\bS_{XY} = (s_{k\ell}, 1\le k \le p, 1 \le \ell \le q)$ using Theorem \ref{thm:HWineq_claim} as its special case where all $\pi$'s are 1. 
We first notice that
$$
s_{k\ell} = \dfrac{1}{n}\sum\limits_{i=1}^n (X_{ik} - \mu_k^X)(Y_{i\ell} - \mu_\ell^Y) - \dfrac{1}{n(n-1)} \sum\limits_{i\neq j} (X_{ik}- \mu_k^X)(Y_{j\ell}- \mu_\ell^Y),
$$
where $\mu_k^X = \mathbb{E} X_{ik}$ and $\mu_\ell^Y = \mathbb{E} Y_{j\ell}$. 
Hence, we can rewrite the sample cross-covariance estimator in a matrix-form by
$$
s_{k\ell} = \bZ_{1(k)}^\top \bA \bZ_{2(\ell)}
$$
where $\bZ_{1(k)}=(X_{1k} - \mu_k^X, \ldots, X_{nk} - \mu_k^X)^\top, \bZ_{2(\ell)}=(Y_{1\ell}-\mu_\ell^Y, \ldots, Y_{n\ell}-\mu_\ell^Y)^\top$, and $\bA = \{n(n-1)\}^{-1} (n \bI - \bone \bone^\top)$. Note that $||\bA||_F = 1/\sqrt{n-1}$ and $||\bA||_2 = 1/(n-1)$.
Then, by Theorem \ref{thm:HWineq_claim}, the element-wise deviation inequality for the sample cross-covariance is
$$
{\rm P} \bigg[
\big| s_{k\ell} - \sigma_{k\ell}^{XY}\big| > t 
\bigg] \le 2\exp \left\{
-  \dfrac{c_1(n-1)t^2}{K_X^2 K_Y^2
}
\right\}, \quad \dfrac{t}{K_X K_Y} < 1,
$$
for some numerical constant $c_1>0$. Putting $t =K_X K_Y \sqrt{\log (2pq/\alpha)/\{c_1(n-1)\}}$ and using the union bounds, we can get 
$$
{\rm P} \left[
\max\limits_{k, \ell} \big| s_{k\ell} - \sigma_{k\ell}^{XY}\big| > C_1 K_X K_Y \sqrt{\dfrac{\log(pq/\alpha) }{n-1}}
\, \right] \le \alpha,
$$
if $n / \log(pq/\alpha) > d_1$ for some numerical constants $C_1, d_1>0$.

\subsection{Missing data case}\label{sec:missing_case}
For the case where data are subject to missing, we introduce assumptions for missing indicators. 
\begin{assumption}\label{assump:miss_ind}
	Each component $\delta_k^X$ of the indicator vector $\bdelta^X = (\delta_k^X, 1\le k \le p)^\top$ corresponding to $\bX$ is 1 if $X_k$ is observed and 0 otherwise. $\bdelta^Y$ is similarly defined. Their moments are given by
	$$
	\begin{array}{l}
	\mathbb{E}\bdelta^X = \bpi^X = (\pi_k^X, 1\le k \le p)^\top, \\
	\mathbb{E}\bdelta^Y = \bpi^Y  = (\pi_\ell^Y, 1\le \ell \le q)^\top,\\
	\mathbb{E} \bdelta^X (\bdelta^Y)^\top = \bpi^{XY} = (\pi_{k\ell}, 1\le k \le p, 1 \le \ell \le q).
	\end{array}
	$$
\end{assumption}
\noindent
Note that the above assumption does not mention independence between components of the indicator vector, which means it allows $\delta_k^X$ and $\delta_\ell^X$, $k\neq \ell$, to be dependent with each other. This implies multiple components in different positions can be missing together under some dependency structure. Assumption 2 of \cite{Park:2021} is equivalent to this, and they called it the general missing dependency assumption. For the missing mechanism, missing completely at random (MCAR) is assumed, that is, $(\bX, \bY)$ is independent of $(\bdelta^X, \bdelta^Y)$. More generally, MCAR can be stated as follows.
\begin{assumption}[Assumption 3 of \cite{Park:2021}]\label{assump:MCAR}
	An event that an observation is missing is independent of both observed and unobserved random variables.
\end{assumption}

Suppose $n$ independent samples are generated from the population model under Assumption \ref{assump:X_and_Y}, \ref{assump:miss_ind}, and \ref{assump:MCAR}. Each sample consists of the observational data $(\bX_i, \bY_i)$ and their missing indicators $(\bdelta_i^X, \bdelta_i^Y)$. However, due to missingness, we can only observe 
$\tilde{X}_{ik} = \delta_{ik}^X X_{ik}$, 
$\tilde{Y}_{j\ell} = \delta_{j\ell}^Y Y_{j\ell}$, for $i,j=1,\ldots, n$, $k=1,\ldots, p$, and $\ell= 1, \ldots, q$. We can easily check that 
$$
\mathbb{E}\big[\sum\limits_{i=1}^n \tilde{X}_{ik} \tilde{Y}_{i\ell} \big] = n \pi_{k \ell} (\sigma_{k\ell}^{XY} + \mu_k^X \mu_\ell^Y), \quad \mathbb{E}\big[\sum\limits_{i\neq j}^n \tilde{X}_{ik} \tilde{Y}_{j\ell} \big] = n(n-1) \pi_k^X\pi_{\ell}^Y \mu_k^X \mu_\ell^Y.
$$
From the above observation, we define an estimator of the cross-covariance as follows, which is unbiased for $\sigma_{k\ell}^{XY}$.
\begin{equation}\label{eq:IPWest_unknown_mean}
	\tilde{s}_{k\ell} = \dfrac{\sum_{i=1}^n \tilde{X}_{ik} \tilde{Y}_{i\ell}}{n \pi_{k \ell}} - \dfrac{\sum_{i\neq j}^n \tilde{X}_{ik} \tilde{Y}_{j\ell}}{n(n-1) \pi_k^X \pi_{\ell}^Y} = (\tilde{\bZ}_{1(k)} * \bdelta_{1(k)})^\top \bA_{k\ell} (\tilde{\bZ}_{2(\ell)} * \bdelta_{2(\ell)}).
\end{equation}
The last representation is a bilinear form of $\tilde{\bZ}_{1(k)}, \tilde{\bZ}_{2(k)}, \bdelta_{1(k)}, \bdelta_{2(k)}$, and $\bA_{k\ell}$  defined as below.
$$
\begin{array}{l}
\tilde{\bZ}_{1(k)}=({X}_{1k}, \ldots, {X}_{nk})^\top, \tilde{\bZ}_{2(\ell)}=({Y}_{1\ell}, \ldots, {Y}_{n\ell})^\top,\\
\bdelta_{1(k)}=(\delta_{1k}^X, \ldots, \delta_{nk}^X)^\top, \bdelta_{2(\ell)}=(\delta_{1\ell}^Y, \ldots, \delta_{n\ell}^Y)^\top,\\
\bA_{k\ell} = \left(
\dfrac{1}{n\pi_{k \ell}} + 
\dfrac{1}{n(n-1)\pi_k^X \pi_\ell^Y}
\right)\bI - 
\dfrac{1}{n(n-1)\pi_k^X \pi_\ell^Y} \bone \bone^\top.
\end{array}
$$
Thus, we apply Theorem \ref{thm:HWineq_claim_nzm} to $\tilde{s}_{k\ell}$ and get, for $t < E_{1,k\ell} / E_{2,k\ell} $
$$
{\rm P} \bigg[
\big| \tilde{s}_{k\ell} - \sigma_{k\ell}^{XY}\big| > t 
\bigg] \le d \exp \left\{
- \dfrac{c_2 t^2}{E_{1,k\ell}}
\right\},
$$
where $c_2, d>0$ are some numerical constants and $E_{1,k\ell}, E_{2,k\ell}$ are defined below.
$$
\begin{array}{l}
E_{1,k\ell} = \max\bigg[
\max\Big\{K_X^2 K_Y^2, K_X^2 |\mu_\ell^Y|^2, |\mu_k^X|^2 K_Y^2,  |\mu_k^X|^2 |\mu_\ell^Y|^2 \Big\} \left(\dfrac{1}{n\pi_{k \ell}^2} + \dfrac{1}{n(n-1)(\pi_k^X)^2 (\pi_\ell^Y)^2}\right),\\
\qquad\qquad
\max\Big\{ (K_X)^2 |\mu_\ell^Y|^2, |\mu_k^X|^2 (K_Y)^2,  |\mu_k^X|^4, |\mu_\ell^Y|^4 \Big\} \dfrac{1}{n}\left(\dfrac{1}{\pi_{k \ell}} - \dfrac{1}{\pi_k^X \pi_\ell^Y}\right)^2
\bigg],\\
E_{2,k\ell} = \max\Big\{K_X K_Y, K_X |\mu_\ell^Y|, |\mu_k^X| K_Y,  |\mu_k^X| |\mu_\ell^Y| \Big\} \left(\dfrac{1}{n\pi_{k \ell}} + \dfrac{1}{n(n-1)\pi_k^X \pi_\ell^Y}\right)
\end{array}
$$

Putting $t = \sqrt{E_{1,k\ell} \log (dpq/\alpha)/ c_2}$ and using the union bound argument to the indices $(k,\ell)$, we can get for some numerical constants $C_2, d_2 >0$
$$
{\rm P} \left[
\max\limits_{k, \ell} \big| \tilde{s}_{k\ell} - \sigma_{k\ell}^{XY}\big| > C_2 \sqrt{\log (pq/\alpha) \max_{k,\ell}E_{1,k\ell}}
\, \right] \le \alpha,
$$
if $\sqrt{\log (pq/\alpha)} < d_2 \min_{k,\ell} \sqrt{E_{1,k\ell}} / E_{2,k\ell}$. 
A simple calculation leads to 
$$
\begin{array}{l}
\max\limits_{k,\ell}E_{1,k\ell} \le \dfrac{\{f_2(K_X, K_Y, \mu_X, \mu_Y)\}^2}{(n-1) (\pi_{\min}^J \wedge (\pi_{\min}^M)^2)^2},\\
\min\limits_{k,\ell} \sqrt{E_{1,k\ell}} / E_{2,k\ell} \ge g_2(K_X, K_Y, \mu_X, \mu_Y) \sqrt{n-1}(\pi_{\min}^J \wedge (\pi_{\min}^M)^2),
\end{array}
$$
where $\pi_{\min}^J = \min_{k,\ell} \pi_{k \ell}$, $\pi_{\min}^M =\min(\min_{k} \pi_k^X, \min_{\ell} \pi_\ell^Y)$, $\mu_X=\max_{k} |\mu_k^X|$, $\mu_Y=\max_{\ell} |\mu_\ell^Y|$, $f_2(K_X, K_Y, \mu_X, \mu_Y) = \max\{
K_XK_Y, \mu_X K_Y, K_X \mu_Y,  \mu_X\mu_Y, \mu_X^2, \mu_Y^2\}$, and $g_2(K_X, K_Y, \mu_X, \mu_Y) = \min\{1, \mu_X/K_Y, \mu_Y/K_X, \mu_X \mu_Y/(K_X K_Y)\}$. The superscript ``J'' and ``M'' stand for joint and marginal, respectively. Then, we conclude that for some numerical constants $\tilde{C}_2, d_2>0$
$$
{\rm P} \left[
\max\limits_{k, \ell} \big| \tilde{s}_{k\ell} - \sigma_{k\ell}^{XY}\big| >  \tilde{C}_2f_2(K_X, K_Y, \mu_X, \mu_Y) \sqrt{	\dfrac{\log (pq/\alpha)}{(n-1)(\pi_{\min}^J \wedge (\pi_{\min}^M)^2)}}
\, \right] \le \alpha, 
$$
if $\dfrac{n-1}{\log (pq/\alpha)} > \dfrac{d_2}{g_2(K_X, K_Y, \mu_X, \mu_Y) \pi_{\min}^J \wedge (\pi_{\min}^M)^2}$ holds.

\subsection{Measurement error case}\label{sec:msr_err_case}

The missing data problem is a special case of the multiplicative measurement error case if the multiplicative factor only takes either 1 or 0. Under some boundedness condition on the factors, we can extend the current framework to the multiplicative measurement error case, which is given below.
\begin{theorem}\label{thm:HWineq_msr_error}
	Let $(Z_{1j}, Z_{2j})$, $j=1,\ldots, n$, be independent pairs of (possibly correlated) random variables satisfying $\mathbb{E} Z_{1j} =\mathbb{E} Z_{2j}=0$, and
	$||Z_{ij}||_{\psi_2} \le K_i$, $i=1,2$ and $\forall j$. Also, suppose $(\gamma_{1j}, \gamma_{2j})$, $j=1,\ldots, n$, be independent pairs of (possibly correlated) non-negative random variables such that $\gamma_{ij} \le B_{ij}$ almost surely for some $B_{ij}>0$, $i=1,2$ and $\forall j$.
	Assume $Z_{ij}$ and $\gamma_{i'j'}$ are independent for distinct pairs $(i,j)\neq (i', j')$. Then, we have that for every $t > 0$
	$$
	\begin{array}{l}
	{\rm P} \bigg[
	\big| (\bZ_1 * \bgamma_1)^\top \bA (\bZ_2 * \bgamma_2) - \mathbb{E}(\bZ_1 * \bgamma_1)^\top \bA (\bZ_2 * \bgamma_2)\big| > t 
	\bigg] \\
	\qquad\qquad\qquad \le 2\exp \left\{
	- c \min \left(\dfrac{t^2}{K_1^2 K_2^2 ||\bD(\bB_1)\bA \bD(\bB_2)||_F^2 }, \dfrac{t}{K_1 K_2 ||\bA||_2}\right)
	\right\},
	\end{array}
	$$
	for some numerical constant $c>0$. $\bD(\bB_1)$ is a diagonal matrix with diagonal elements being from $\bB_1=(B_{1i}, 1\le i \le n)^\top$. $\bD(\bB_2)$ is similarly defined.
\end{theorem}
\noindent
The proof is straightforward, and thus we outline the idea behind it and omit the detail. In the diagonal part, we need to modify (\ref{eq:mgf_diag_sum}) as
$$
\mathbb{E} (\lambda a_{ii} \gamma_{1i}\gamma_{2i} Z_{1i}Z_{2i})^s \le 
\mathbb{E} \big(\lambda |a_{ii}| B_{1i}B_{2i} |Z_{1i}Z_{2i}| \big)^s, \quad \lambda > 0.
$$
For the off-diagonal case, a careful investigation into its proof shows that the missing indicators are conditioned in the analysis and the fact they are Bernoulli random variables is not used until Step 4 in Appendix \ref{app:proof_lem_cross_product_mgf}. The result of Lemma \ref{lem:step4} (see Step 4 in Appendix \ref{app:proof_lem_cross_product_mgf}) can be extended to the bounded random errors as we can derive 
$$
\mathbb{E} \left[ \exp\left(\tau \sum\limits_{i: \eta_i = 1}\gamma_{1i} \sum\limits_{j:\eta_j=0} a_{ij}^2\gamma_{2j} \right) | \eta \right] \le
\exp\left(
\tau \sum\limits_{i \neq j} a_{ij}^2 B_{1i} B_{2i}
\right), \quad \tau > 0.
$$
Furthermore, we state the result for the case with non-zero means.
\begin{theorem}\label{thm:HWineq_msr_error_nzm}
	Let $(\tilde{Z}_{1j}, \tilde{Z}_{2j})$, $j=1,\ldots, n$, be independent pairs of (possibly correlated) random variables with $\mathbb{E} \tilde{Z}_{1i}=\mu_{1i}$, $\mathbb{E} \tilde{Z}_{2j}=\mu_{2j}$, and
	$||\tilde{Z}_{ij} - \mu_{ij} ||_{\psi_2} \le K_i$, $i=1,2$ and $\forall j$. Also, suppose $(\gamma_{1j}, \gamma_{2j})$ be pairs of (possibly correlated) non-negative random variables such that $\gamma_{ij} \le B_{ij}$ almost surely for some $B_{ij}>0$, $i=1,2$ and $\forall j$. Assume $\tilde{Z}_{ij}$ and $\gamma_{i'j'}$ are independent for distinct pairs $(i,j)\neq (i', j')$. Then, we have that for every $t > 0$
	$$
	\begin{array}{l}
	{\rm P} \bigg[
	\big| (\tilde{\bZ}_1 * \bgamma_1)^\top \bA (\tilde{\bZ}_2 * \bgamma_2) - \mathbb{E}(\tilde{\bZ}_1 * \bgamma_1)^\top \bA (\tilde{\bZ}_2 * \bgamma_2)\big| > t 
	\bigg] \le d\exp \left\{
	- c \min \left(\dfrac{t^2}{E_1}, \dfrac{t}{E_2}\right)
	\right\},
	\end{array}
	$$
	for some numerical constants $c,d>0$. $E_1$ and $E_2$ are defined by
	$$
	\begin{array}{lll}
	E_1 = \max\Bigg\{ & K_1^2 K_2^2 ||\bD(\bB_1)\bA \bD(\bB_2)||_F^2, &  \max\limits_{1\le i \le n} B_{1i}^2 K_2^2 ||\bD(\bmu_1 * \bmu_1)^{1/2}\bA \bD(\bB_2)^{1/2}||_{F}^2,\\
	& \max\limits_{1\le i \le n} B_{2i}^2 K_1^2 ||\bD(\bB_1)^{1/2}\bA \bD(\bmu_2 * \bmu_2)^{1/2}||_{F}^2,
	& K_2^2 ||\{\bA^\top (\bmu_1 * \bu_1)\}*\bB_2||_2^2, \\
	& K_1^2 ||\{\bA(\bmu_2 * \bu_2)\} * \bB_1||_2^2, & \max\limits_{1\le i \le n} B_{1i}^2 \max\limits_{1\le i \le n} B_{2i}^2||\bD(\bmu_1)\bA \bD(\bmu_2)||_{F}^2\Bigg\},\\[1em]
	E_2 = \max\Bigg\{  & K_1 K_2 ||\bA||_2, &  \max\limits_{1\le i \le n} B_{1i} K_2 ||\bD(\bmu_1)\bA ||_2, \\
	& & \quad \max\limits_{1\le i \le n} B_{2i} K_1 ||\bA \bD(\bmu_2) ||_2 \Bigg\}. 
	\end{array}
	$$
\end{theorem}

The rest of arguments are similar to Section \ref{sec:missing_case}. Assume we observe $\tilde{\bX}_i = \bdelta_i^X * \bX_i$ and $\tilde{\bY}_i = \bdelta_i^Y * \bY_i$, $i=1,\ldots, n$. While $(\bX_i, \bY_i)$ is an independent copy of $(\bX, \bY)$ in Assumption \ref{assump:X_and_Y}, $(\bdelta_i^X, \bdelta_i^Y)$ is no longer a vector of binary variables but an independent copy of $(\bdelta^X, \bdelta^Y)$ in Assumption \ref{assump:msr_err}.
\begin{assumption}\label{assump:msr_err}
	Each component $\delta_k^X$ of $\bdelta^X = (\delta_k^X, 1\le k \le p)^\top$ is a measurement error corresponding to $X_k$ of $\bX$, which is a non-negative random variable satisfying $\delta_k^X \le B_k^X$ almost surely for each $k$. $\bdelta^Y$ is similarly defined. Their moments are given by
	$$
	\begin{array}{l}
	\mathbb{E}\bdelta^X = \bu^X = (u_k^X, 1\le k \le p)^\top, \\
	\mathbb{E}\bdelta^Y = \bu^Y  = (u_\ell^Y, 1\le \ell \le q)^\top,\\
	\mathbb{E} \bdelta^X (\bdelta^Y)^\top = \bU^{XY} = (u_{k\ell}, 1\le k \le p, 1 \le \ell \le q).
	\end{array}
	$$
\end{assumption}
Accordingly, the unbiased estimator for $\sigma_{k\ell}^{XY}$ is 
$$
\check{s}_{k\ell} = \dfrac{\sum_{i=1}^n \tilde{X}_{ik} \tilde{Y}_{i\ell}}{n u_{k \ell}} - \dfrac{\sum_{i\neq j}^n \tilde{X}_{ik} \tilde{Y}_{j\ell}}{n(n-1) u_k^X u_{\ell}^Y}.
$$
In this case, we can derive 
$$
{\rm P} \bigg[
\big| \check{s}_{k\ell} - \sigma_{k\ell}^{XY}\big| > t 
\bigg] \le d \exp \left\{
- \dfrac{c_3 t^2}{E_{1,k\ell}}
\right\}, \quad t < E_{1,k\ell} / E_{2,k\ell}
$$
where $c_3, d>0$ are some numerical constants and $E_{1,k\ell}, E_{2,k\ell}$ are defined below.
$$
\begin{array}{l}
E_{1,k\ell} = \max\bigg[
\max\Big\{K_X^2 K_Y^2 (B_k^X)^2 (B_\ell^Y)^2, 
|\mu_k^X|^2 K_Y^2 (B_k^X)^2 B_\ell^Y,\\
\qquad\qquad\qquad\qquad
K_X^2 |\mu_\ell^Y|^2  B_k^X (B_\ell^Y)^2, 
|\mu_k^X|^2  |\mu_\ell^Y|^2 (B_k^X)^2 (B_\ell^Y)^2\Big\}
\left(\dfrac{1}{n u_{k \ell}^2} + \dfrac{1}{n(n-1)(u_k^X)^2 (u_\ell^Y)^2}\right),\\
\qquad\qquad\qquad\qquad
\max\Big\{ K_X^2 |\mu_\ell^Y|^2  (B_k^X)^2 (u_\ell^Y)^2, |\mu_k^X|^2 K_Y^2 (u_k^X)^2 (B_\ell^Y)^2\Big\} \dfrac{1}{n}\left(\dfrac{1}{u_{k \ell}} - \dfrac{1}{u_k^X u_\ell^Y}\right)^2
\bigg],\\
E_{2,k\ell} = \max\Big\{K_X K_Y, K_X |\mu_\ell^Y| B_k^X, |\mu_k^X|K_Y B_\ell^X \Big\} \left(\dfrac{1}{n u_{k \ell}} + \dfrac{1}{n(n-1) u_k^X u_\ell^Y}\right).
\end{array}
$$
Moreover, we can observe 
$$
\begin{array}{l}
\max\limits_{k,\ell}E_{1,k\ell} \le \dfrac{\{f_3(K_X, K_Y, \mu_X, \mu_Y, B_X, B_Y)\}^2}{(n-1) (u_{\min}^J \wedge (u_{\min}^M)^2)^2},\\
\min\limits_{k,\ell} \sqrt{E_{1,k\ell}} / E_{2,k\ell} \ge g_3(K_X, K_Y, \mu_X, \mu_Y, B_X, B_Y) \sqrt{n-1} \cdot u_{\min}^J \wedge (u_{\min}^M)^2,
\end{array}
$$
where $u_{\min}^J = \min_{k,\ell} u_{k \ell}$, $u_{\min}^M =\min(\min_{k} u_k^X, \min_{\ell} u_\ell^Y)$, $\mu_X=\max_{k} |\mu_k^X|$, $\mu_Y=\max_{\ell} |\mu_\ell^Y|$, $B_X=\max_{k} B_k^X$, $B_Y=\max_{\ell} B_\ell^Y$, $g_3(K_X, K_Y, \mu_X, \mu_Y, B_X, B_Y) = \min\{1, K_X / (B_X \mu_X), K_Y / (B_Y \mu_Y)\}$, and
$$
\begin{array}{l}
f_3(K_X, K_Y, \mu_X, \mu_Y, B_X, B_Y) = \max\{ K_X K_Y B_X B_Y, 
\mu_X K_Y B_X B_Y, K_X \mu_Y B_X B_Y,\\ 
\qquad\qquad\qquad\qquad\qquad\qquad\qquad\qquad\qquad \mu_X \mu_Y B_X B_Y, 
K_X \mu_Y B_X u_Y, \mu_X K_Y u_X B_Y,\}.
\end{array}
$$
The superscript ``J'' and ``M'' stand for joint and marginal, respectively.

Repeating the calculation as in the previous section, we conclude that  for some numerical constants $C_3, d_3 >0$
$$
{\rm P} \left[
\max\limits_{k, \ell} \big| \check{s}_{k\ell} - \sigma_{k\ell}^{XY}\big| >  C_3 f_3(K_X, K_Y, \mu_X, \mu_Y, B_X, B_Y) \sqrt{	\dfrac{\log (pq/\alpha)}{(n-1) \cdot (u_{\min}^J)^2 \wedge (u_{\min}^M)^4}}
\, \right] \le \alpha,
$$
if $\dfrac{n-1}{\log (pq/\alpha)} \ge \dfrac{d_3}{g_3(K_X, K_Y, \mu_X, \mu_Y, B_X, B_Y) \cdot u_{\min}^J \wedge (u_{\min}^M)^2}$ holds.

\section{Discussion}

We discuss the generalized Hanson-Wright inequality where the sparse structure and bilinear form are considered for the first time. This extension facilitates an analysis of concentration of the sample (cross-)covariance estimator even when mean parameters are unknown and some of the data are missing. We apply this result to multiple testing of the cross-covariance matrix. 

We further consider a measurement error case extended from the missing data case, which is limited to a bounded random variable. It would be interesting to consider more general multiplicative errors as future work; for example, the truncated normal distribution defined on $(0, \infty)$ can be a good example.

\bibliographystyle{apalike}
\bibliography{references}

\newpage
\appendix

\section*{Appendix}

\section{Proof of Theorem 1}\label{app:proof_HWineq}

\subsection{Proof of Lemma \ref{lem:HWineq_diag} (diagonal part)}

Define $S_0$ by
$$
S_0 = \sum_{i=1}^n a_{ii} \gamma_{1i}\gamma_{2i} X_{1i}Z_{2i} - 
\mathbb{E} \left[\sum_{i=1}^n a_{ii} \gamma_{1i}\gamma_{2i} Z_{1i}Z_{2i} \right].
$$
For $|\lambda|< 1/(4e K_1 K_2 \max_i |a_{ii}|)$, 
\begin{equation}\label{eq:mgf_diag_sum}
	\begin{array}{l}
		\mathbb{E} \exp(\lambda S_0) \\[0.8em]
		= \dfrac{\mathbb{E} \exp(\lambda \sum_{i=1}^n a_{ii} \gamma_{1i}\gamma_{2i} Z_{1i}Z_{2i})}{\exp\left(\lambda
			\mathbb{E} \left[\sum_{i=1}^n a_{ii} \gamma_{1i}\gamma_{2i} Z_{1i}Z_{2i} \right]
			\right)} \\[0.8em]
		= \prod\limits_{i=1}^n\dfrac{\mathbb{E}_\gamma\mathbb{E}_X \exp(\lambda  a_{ii} \gamma_{1i}\gamma_{2i} Z_{1i}Z_{2i})}{\exp\left(
			\lambda a_{ii} \mathbb{E}\gamma_{1i}\gamma_{2i} \mathbb{E}Z_{1i}Z_{2i}
			\right)} \\[0.8em]
		= \prod\limits_{i=1}^n\dfrac{(1 - \pi_{12,i})  + \pi_{12,i} \mathbb{E}_X \exp(\lambda  a_{ii} Z_{1i}Z_{2i})}{\exp\left(
			\lambda a_{ii} \pi_{12,i} \mathbb{E}Z_{1i}Z_{2i}
			\right)} \\[0.8em]
		\le \prod\limits_{i=1}^n\dfrac{(1 - \pi_{12,i})  + \pi_{12,i}
			(1 + \lambda a_{ii} \mathbb{E}Z_{1i}Z_{2i} + 
			16\lambda^2 a_{ii}^2 K_1^2 K_2^2)
		}{\exp\left(
			\lambda a_{ii} \pi_{12,i} \mathbb{E}Z_{1i}Z_{2i}
			\right)} \\[0.8em]
		\le \prod\limits_{i=1}^n\dfrac{
			\exp\left\{ \pi_{12,i}(\lambda a_{ii} \mathbb{E}Z_{1i}Z_{2i} + 
			16\lambda^2 a_{ii}^2 K_1^2 K_2^2)\right\}
		}{\exp\left(
			\lambda a_{ii} \pi_{12,i} \mathbb{E}Z_{1i}Z_{2i}
			\right)} \\[0.8em]
		\le \exp\left\{ 
		16\lambda^2 K_1^2 K_2^2 \sum\limits_{i=1}^n \pi_{12,i} a_{ii}^2\right\}
	\end{array}
\end{equation}
where the first inequality uses Lemma \ref{lem:subexp_mgf} and the second holds since $1+x \le e^x$.
\begin{lemma}\label{lem:subexp_mgf}
	Assume that $||Z_{1i}||_{\psi_2} \le K_1, ||Z_{2i}||_{\psi_2} \le K_2$ for some $K_1, K_2>0$ and $|\lambda|< 1/(4eK_1 K_2 \max_i |a_{ii}|)$ for given $\{a_{ii}\}_{i=1}^n \subset \mathbb{R}$. Then, for any $i$, we have 
	$$
	\mathbb{E} \exp\left(
	\lambda a_{ii} Z_{1i} Z_{2i}
	\right) - 1 \le 
	\lambda a_{ii} \mathbb{E} Z_{1i} Z_{2i} + 16 \lambda^2 a_{ii}^2 K_1^2 K_2^2.
	$$
\end{lemma}
\begin{proof}
	First, we define the sub-exponential (or $\psi_1$-) norm of a random variable by
	$$
	||X||_{\psi_1} = \sup_{p \ge 1} \dfrac{(\mathbb{E}|X|^p)^{1/p}}{p}.
	$$
	Since the product of sub-Gaussian random variables is sub-exponential, we define $Y_i = Z_{1i}Z_{2i}/||Z_{1i}Z_{2i}||_{\psi_1}$. Setting $t_i = \lambda a_{ii} ||Z_{1i}Z_{2i}||_{\psi_1}$, we get
	$$
	\begin{array}{l}
	\mathbb{E} \exp\left(
	t_i Y_i
	\right)\\
	= 1 + t_i \mathbb{E}Y_i + \sum_{s\ge 2} \dfrac{t_k^s\mathbb{E}Y_i^s }{s!}\\
	\le 	1 + t_i \mathbb{E}Y_i + \sum_{s\ge 2} \dfrac{|t_k|^s\mathbb{E}|Y_i|^s }{s!}\\
	\le 	1 + t_i \mathbb{E}Y_i + \sum_{s\ge 2} \dfrac{2|t_k|^s s^s }{s!} \\
	\le 	1 + t_i \mathbb{E}Y_i + \sum_{s\ge 2} \dfrac{|t_k|^s e^s }{\sqrt{\pi}} 
	\end{array}
	$$
	In the last two inequalities, we use the following observations. First, note that for the subexponential variable satisfying $||Y_i||_{\psi_1} = 1$, it holds (see Prop 2.7.1, \cite{Vershynin:2018}) 
	$$
	\mathbb{E}|Y_i|^s \le 2 s^s, \quad s\ge 1.
	$$
	Second, we use Stirling's formula for $s \ge 2$ that
	$$
	\dfrac{1}{s!} \le \dfrac{e^s}{ 2s^s \sqrt{\pi}}.
	$$
	If $|\lambda|< 1/(4eK_1 K_2 \max_i |a_{ii}|)$, then $e|t_i| \le 1/2$ and thus we get
	$$
	\sum_{s\ge 3} e^s |t_i|^s   \le (e|t_i|)^3 \sum_{s\ge 0}(1/2)^s \le (e|t_i|)^2.
	$$
	Using the above, we have
	$$
	\begin{array}{rcl}
	\mathbb{E} \exp\left(
	t_i Y_i
	\right) &\le& 
	1 + t_i \mathbb{E}Y_i + 2e^2 |t_i|^2/\sqrt{\pi}\\
	& = & 1 + \lambda t_i \mathbb{E}Z_{1i}Z_{2i} + 2e^2 \lambda^2 a_{ii}^2 ||Z_{1i}Z_{2i}||_{\psi_1}^2/\sqrt{\pi}\\
	& \le & 1 + \lambda t_i \mathbb{E}Z_{1i}Z_{2i} + 16 \lambda^2 a_{ii}^2 K_1^2 K_2^2.
	\end{array}
	$$
\end{proof}

Then, for $x>0$ and $0<\lambda <  1/(4e K_1 K_2 \max_i |a_{ii}|)$, we have
$$
\begin{array}{l}
{\rm P}(S_0 > x) \\
= {\rm P}\left(\exp(\lambda S_0) > \exp(\lambda x)\right) \\
\le \dfrac{\mathbb{E} \exp(\lambda S_0)}{\exp(\lambda x)} \\
\le \exp\left\{ -\lambda x + 
16\lambda^2 K_1^2 K_2^2 \sum\limits_{i=1}^n \pi_{12,i} a_{ii}^2\right\}.
\end{array}
$$
For the optimal choice of $\lambda$, that is, 
$$
\lambda = \min\left(
\dfrac{x}{32e K_1^2 K_2^2 \sum_i \pi_{12,i} a_{ii}^2},
\dfrac{1}{4e K_1 K_2 \max_i |a_{ii}|}
\right)
$$
we can obtain the concentration bound
$$
{\rm P}(S_0 > x) \le 
\exp \left\{
-\min\left(
\dfrac{x^2}{32K_1^2 K_2^2 \sum_i \pi_{12,i} a_{ii}^2},
\dfrac{x}{8e K_1 K_2 \max_i |a_{ii}|}
\right)
\right\}.
$$
Considering $-A_{\rm diag}$ instead of $A_{\rm diag}$ in the theorem, we have the same probability bound for ${\rm P}(S_0 < -x)$, $x>0$, which concludes the proof.

\subsection{Proof of Lemma \ref{lem:HWineq_offdiag} (off-diagonal part)}
Define $S_{\rm off}$ by
$$
S_{\rm off} = \sum_{1 \le i \neq j \le n} a_{ij} \gamma_{1i}\gamma_{2j} Z_{1i}Z_{2j},
$$
whose expectation is zero due to independence across distinct $i$ and $j$. Then, we claim the lemma below:
\begin{lemma}\label{lem:cross_product_mgf}
	Assume $||Z_{1i}||_{\psi_2}=||Z_{2i}||_{\psi_2} = 1$ and let $\{a_{ij}\}_{1 \le i \neq j \le n} \subset \mathbb{R}$ be given. For $|\lambda| < 1/(2 \sqrt{C_4} || A||_2)$ for some numeric constant $C_4>0$,  we have 
	$$
	\mathbb{E} \exp\left(
	\lambda S_{\rm off}
	\right) \le 
	\exp \left( 1.44 C_4 \lambda^2 \sum_{i\neq j} a_{ij} \pi_{1i}\pi_{2j} \right).
	$$
\end{lemma}
\noindent
The proof is pending until Appendix \ref{app:proof_lem_cross_product_mgf}. Without loss of generality, we can assume $||Z_{1i}||_{\psi_2}=||Z_{2i}||_{\psi_2} = 1$; otherwise set $Z_{1i} \leftarrow Z_{1i}/||Z_{1i}||_{\psi_2}, Z_{2i} \leftarrow Z_{2i} / ||Z_{2i}||_{\psi_2}$. Using Lemma \ref{lem:cross_product_mgf}, we get for $x>0$ and $0<\lambda < 1/(2 \sqrt{C_4} || A||_2)$
$$
\begin{array}{l}
{\rm P}(S_0 > x) \\
= {\rm P}\left(\exp(\lambda S_0) > \exp(\lambda x)\right) \\
\le \dfrac{\mathbb{E} \exp(\lambda S_0)}{\exp(\lambda x)} \\
\le \exp\left\{ -\lambda x + 
1.44 C_4 \lambda^2 \sum_{i\neq j} a_{ij} \pi_{1i}\pi_{2j} \right\}.
\end{array}
$$
For the optimal choice of $\lambda$, that is, 
$$
\lambda = \min\left(
\dfrac{x}{2.88 C_4 \sum_{i\neq j} a_{ij} \pi_{1i}\pi_{2j}},
\dfrac{1}{2 \sqrt{C_4} || A||_2}
\right)
$$
we can obtain the concentration bound
$$
{\rm P}(S_0 > x) \le 
\exp \left\{
-c\min\left(
\dfrac{x^2}{\sum_{i\neq j} a_{ij} \pi_{1i}\pi_{2j}},
\dfrac{x}{|| A||_2}
\right)
\right\}.
$$
Considering $-A_{\rm off}$ instead of $A_{\rm off}$ in the theorem, we have the same probability bound for ${\rm P}(S_0 < -x)$, $x>0$, which concludes the proof.

\subsection{Proof of Lemma \ref{lem:cross_product_mgf}}\label{app:proof_lem_cross_product_mgf}
This proof follows the logic of the proof of (9) in \cite{Zhou:2019}. We first decouple the off-diagonal sum $S_{\rm off}$.

\subsubsection*{Step 1. Decoupling}
We introduce Bernoulli variables $\eta = (\eta_1, \ldots, \eta_n)^\top$ with $\mathbb{E}\eta_k = 1/2$ for any $k$. They are independent with each other and also independent of $Z_1, Z_2$ and $\gamma_1, \gamma_2$. Given $\eta$, we define $Z_1^{\eta}$ by a subvector of $Z_1$ at which $\eta_i=1$ and $Z_2^{1-\eta}$ by a subvector of $Z_2$ at which $\eta_j=0$, respectively. Let $\mathbb{E}_Q$ be the expectation with respect to a random variable $Q$ where $Q$ can be any of $Z_1, Z_2, \gamma_1, \gamma_2, \eta$ in our context, or $X, \gamma$ to denote $Z_1, Z_2$ and $\gamma_1, \gamma_2$ together. Define a random variable 
$$
S_\eta = \sum_{1 \le i, j \le n} a_{ij} \eta_{i}(1-\eta_{j})\gamma_{1i}\gamma_{2j} Z_{1i}Z_{2j}.
$$
Using $S_{\rm off} = 4\mathbb{E}_\gamma S_\eta$ and Jensen's inequality with $x \mapsto e^{ax}$ for any $a \in \mathbb{R}$, we get
$$
\mathbb{E} \exp(\lambda S_{\rm off}) = 	\mathbb{E}_\gamma  \mathbb{E}_X \exp(\mathbb{E}_\eta 4\lambda S_\eta) \le 
\mathbb{E}_\gamma  \mathbb{E}_X \mathbb{E}_\eta \exp(4\lambda S_\eta).
$$
We condition all the variables except $Z_2^{1-\eta} = (Z_{2j}, \eta_j = 0)$ on $\exp(4\lambda S_\eta)$ and consider its moment generating function denoted by $f(\gamma_1, \gamma_2, \eta, Z_1^\eta)$
$$
f(\gamma_1, \gamma_2, \eta, Z_1^\eta) = \mathbb{E}\left(
\exp(4\lambda S_\eta) | \gamma_1, \gamma_2, \eta, Z_1^\eta
\right).
$$
Note that $S_\eta$ can be seen a linear combination of sub-Gaussian variables $Z_{2j}$, for $j$ such that $\eta_j = 0$, that is,
$$
S_\eta = \sum_{j: \eta_j = 0}Z_{2j} \left(\sum_{i: \eta_i = 1} a_{ij} \gamma_{1i}\gamma_{2j} Z_{1i}\right),
$$
conditional on all other variables. Thus, the conditional distribution of $S_\eta$ is sub-Gaussian satisfying
$$
||S_\eta||_{\psi_2} \le C_0 \sigma_{\eta, \gamma},
$$
where $\sigma_{\eta, \gamma}^2 = \sum\limits_{j: \eta_j = 0}\gamma_{2j} \left(\sum\limits_{i: \eta_i = 1} a_{ij} \gamma_{1i} Z_{1i}\right)^2$.
Therefore, we have that for some $C'>0$
\begin{equation}\label{eq:step1_mgf_bound}
	f(\gamma_1, \gamma_2, \eta, Z_1^\eta) = \mathbb{E}\left(
	\exp(4\lambda S_\eta) | \gamma_1, \gamma_2, \eta, Z_1^\eta
	\right) \le \exp(C' \lambda^2 \sigma_{\eta, \gamma}^2).
\end{equation}
Taking the expectations on both sides, we get
$$
\mathbb{E}_\gamma \mathbb{E}_{Z_1^\eta} f(\gamma_1, \gamma_2, \eta, Z_1^\eta) \le 
\mathbb{E}_\gamma \mathbb{E}_{Z_1^\eta} \exp(C' \lambda^2 \sigma_{\eta, \gamma}^2) =: \tilde{f}_{\eta}.
$$
\subsubsection*{Step 2. Reduction to normal random variables}
Assume that $\eta$, $\gamma_1, \gamma_2$, and $Z_1^\eta$ are fixed. Let $g=(g_1, \ldots, g_n)^\top$ be given where $g_i$ is i.i.d. from $N(0,1)$. Replacing $Z_{2j}$ by $g_j$ in $S_\eta$, we define a random variable $Z$ by
$$
Z = \sum_{j: \eta_j = 0}g_j \left(\sum_{i: \eta_i = 1} a_{ij} \gamma_{1i}\gamma_{2j} Z_{1i}\right).
$$
Due to the property of Gaussian variables, the conditional distribution of $Z$ follows $N(0, \sigma_{\eta, \gamma}^2)$. Hence, its conditional moment generating function is given by
\begin{equation}\label{eq:step2_mgf_bound}
	\mathbb{E}_g \exp(t Z) = \mathbb{E} \left[\exp(t Z) | \eta, \gamma_1, \gamma_2, Z_1^\eta \right] = \exp(t^2 \sigma_{\eta, \gamma}^2 / 2).
\end{equation}
Now, consider $Z$ as a linear combination of $\{Z_{1i}: \eta_i = 1\}$ by rewriting it by
$$
Z = \sum_{i: \eta_i = 1} Z_{1i} \left(\sum_{j: \eta_j = 0}g_j a_{ij} \gamma_{1i}\gamma_{2j} \right),
$$
where $Z_{1i}$'s are regarded random variables and others fixed. Then, we can get for some $C_3>0$
\begin{equation}\label{eq:step2_mgf_normal}
	\mathbb{E}_{Z_1^\eta} \exp(\sqrt{2C'}\lambda Z) = \mathbb{E} \left[\exp(\sqrt{2C'}\lambda Z) | g, \eta, \gamma_1, \gamma_2
	\right]  \le \exp\left(
	C_3 \lambda^2 \sum_{i: \eta_i=1} \gamma_{1i} \left(\sum_{j: \eta_j = 0}g_j a_{ij} \gamma_{2j} \right)^2
	\right).
\end{equation}
We now get an upper bound of $\tilde{f}_{\eta}$ from step 1 by using (\ref{eq:step2_mgf_bound}) and (\ref{eq:step2_mgf_normal}). First, choosing $t=\sqrt{2C'}\lambda$ at (\ref{eq:step2_mgf_bound}) gives the same term in (\ref{eq:step1_mgf_bound})
$$
\tilde{f}_{\eta} = \mathbb{E}_\gamma \mathbb{E}_{Z_1^\eta} \mathbb{E}_g \exp(\sqrt{2C'}\lambda Z) = \mathbb{E}_\gamma \mathbb{E}_g  \mathbb{E}_{Z_1^\eta}  \exp(\sqrt{2C'}\lambda Z).
$$
Then, we apply (\ref{eq:step2_mgf_normal}) to the right-hand side of the above:
\begin{equation}\label{eq:step2_final}
	\tilde{f}_{\eta} \le 
	\mathbb{E}_\gamma \mathbb{E}_g \exp\left(
	C_3 \lambda^2 \sum_{i: \eta_i=1} \gamma_{1i} \left(\sum_{j: \eta_j = 0}g_j a_{ij} \gamma_{2j} \right)^2 \right) = 
	\mathbb{E} \left[ \exp\left(
	C_3 \lambda^2 (g^\top B_{\eta, \gamma} g) \right) | \eta
	\right],
\end{equation}

where $B_{\eta, \gamma}$ has its $(k,\ell)$-th element $\sum\limits_{i: \eta_i = 1}\gamma_{1i} a_{ik} a_{i\ell} \gamma_{2k}\gamma_{2\ell}$ if $k,\ell \in \{j: \eta_j = 0\}$; $0$, otherwise. Note that $B_{\eta, \gamma}$ is symmetric.

\subsubsection*{Step 3. Integrating out the normal random variables}	
For fixed $\gamma_1, \gamma_2, \eta$, the quadratic form $g^\top B_{\eta, \gamma} g$ is identical in distribution with $\sum_{j=1}^n s_j^2 g_j^2$ due to the rotational invariance of $g$ where $s_j^2$ is the eigenvalue of $B_{\eta, \gamma}$. Note that the eigenvalues satisfy, for any realization of $\gamma_1, \gamma_2, \eta$,
\begin{equation}\label{eq:step3_eigenvalues}
	\begin{array}{rcl}
		\max\limits_{1\le j \le n} s_j^2 &= &||B_{\eta, \gamma}||_2 \le ||A||_2^2\\
		\sum\limits_{1\le j \le n} s_j^2 &= &\tr(B_{\eta, \gamma}) = \sum\limits_{i: \eta_i = 1}\gamma_{1i} \sum\limits_{j:\eta_j=0} a_{ij}^2\gamma_{2j}.
	\end{array}
\end{equation}
Now, using $g_j^2 \sim \chi^2_1$ and $\mathbb{E}\exp(t g_j^2) \le (1-2t)^{-1/2} \le \exp(2t)$ for $t < 1/4$, we get from (\ref{eq:step2_final}) and (\ref{eq:step3_eigenvalues})
$$
\begin{array}{l}
\tilde{f}_{\eta} \le \mathbb{E} \left[ \exp\left(
C_3 \lambda^2 (g^\top B_{\eta, \gamma} g) \right) | \eta
\right] \\[0.5em]
= \mathbb{E}_\gamma \mathbb{E} \left[ \exp\left(
C_3 \lambda^2 (g^\top B_{\eta, \gamma} g) \right) | \eta, \gamma_1, \gamma_2
\right] \\[0.5em]
=  \mathbb{E}_\gamma \mathbb{E} \left[ \exp\left(
\sum\limits_{j=1}^n C_3 \lambda^2 s_j^2 g_j^2 \right) | \eta, \gamma_1, \gamma_2
\right] \\[0.5em]
\le  \mathbb{E}\left[ \prod\limits_{j=1}^n \exp(2C_3 \lambda^2 s_j^2)| \eta \right]\\[0.5em]
=  \mathbb{E} \left[ \exp\left(2C_3 \lambda^2  \sum\limits_{i: \eta_i = 1}\gamma_{1i} \sum\limits_{j:\eta_j=0} a_{ij}^2\gamma_{2j} \right) | \eta \right]
\end{array}
$$
for $\lambda^2 < 1/(4C_3 ||A||_2^2) (< 1/(4C_3 \max_j s_j^2))$. It is worth noting that the Bernoulli variables $\gamma_1, \gamma_2$ are decoupled by $\eta$.

\subsubsection*{Step 4. Integrating out the Bernoulli random variables}	
We now integrate out $\gamma_1, \gamma_2$ from $\tilde{f}_\eta$ by using the following lemma. For $0 < \lambda^2 \le 1/(8C_3 || A||_2^2)$, we have
$$
\tilde{f}_{\eta} \le  \mathbb{E} \left[ \exp\left(2C_3 \lambda^2  \sum\limits_{i: \eta_i = 1}\gamma_{1i} \sum\limits_{j:\eta_j=0} a_{ij}^2\gamma_{2j} \right) | \eta \right] \le \exp\left(
2.88 C_3 \lambda^2 \sum\limits_{i \neq j} a_{ij}^2 \pi_{1i} \pi_{2i}
\right)
$$
\begin{lemma}\label{lem:step4}
	For $0 < \tau \le 1/(4 || A||_2^2)$, the conditional expectation satisfies
	$$
	\mathbb{E} \left[ \exp\left(\tau \sum\limits_{i: \eta_i = 1}\gamma_{1i} \sum\limits_{j:\eta_j=0} a_{ij}^2\gamma_{2j} \right) | \eta \right] \le
	\exp\left(
	1.44 \tau \sum\limits_{i \neq j} a_{ij}^2 \pi_{1i} \pi_{2i}
	\right).
	$$
\end{lemma}
\begin{proof}
	Due to the independence of $\gamma_{1i}, i=1,2,\ldots, n$, we get
	$$
	\begin{array}{l}
	\mathbb{E} \left[ \exp\left(\tau \sum\limits_{i: \eta_i = 1}\gamma_{1i} \sum\limits_{j:\eta_j=0} a_{ij}^2\gamma_{2j} \right) | \gamma_2^{1-\eta}, \eta \right]\\
	= \prod\limits_{i: \eta_i = 1}  \mathbb{E} \left[   \exp\left(\tau \gamma_{1i} \sum\limits_{j:\eta_j=0} a_{ij}^2\gamma_{2j} \right) | \gamma_2^{1-\eta}, \eta \right]	\\
	= \prod\limits_{i: \eta_i = 1}  \left[(1- \pi_{1i}) +  \pi_{1i} \exp\left(\tau \sum\limits_{j:\eta_j=0} a_{ij}^2\gamma_{2j} \right)  \right].
	\end{array}
	$$
	Then, we use the local approximation of the exponential function:
	$$
	e^x - 1 \le 1.2 x, \quad 0 \le x \le 0.35.
	$$
	To use it, $\tau$ should satisfies, for given $\eta$ and $\gamma_2$,
	\begin{equation}\label{eq:step4_tau_cond}
		\tau \sum\limits_{j:\eta_j=0} a_{ij}^2\gamma_{2j}  < 0.35,
	\end{equation}
	which leads to 
	$$
	(1- \pi_{1i}) +  \pi_{1i} \exp\left(\tau \sum\limits_{j:\eta_j=0} a_{ij}^2\gamma_{2j}\right) \le 
	1 + 1.2 \pi_{1i} \tau \sum\limits_{j:\eta_j=0} a_{ij}^2\gamma_{2j} \le
	\exp\left(
	1.2 \pi_{1i} \tau \sum\limits_{j:\eta_j=0} a_{ij}^2\gamma_{2j}
	\right),
	$$
	where we use $1 + x \le e^x$ in the last inequality. A sufficient condition for $\tau$ for any realization of $\eta$ and $\gamma_2$ in (\ref{eq:step4_tau_cond}) is $\tau \le 1/(4||A||_2^2)$ since
	$$
	\tau \sum\limits_{j:\eta_j=0} a_{ij}^2\gamma_{2j} \le \tau \max_{i} \sum_j a_{ij}^2 \le \tau ||A||_2^2 \le 0.25.
	$$
	Hence, for $\tau \le 1/(4||A||_2^2)$,
	$$
	\begin{array}{l}
	\mathbb{E} \left[ \exp\left(\tau \sum\limits_{i: \eta_i = 1}\gamma_{1i} \sum\limits_{j:\eta_j=0} a_{ij}^2\gamma_{2j} \right) |  \eta \right] \\
	= \mathbb{E}_{\gamma_2^{1-\eta}} \mathbb{E} \left[ \exp\left(\tau \sum\limits_{i: \eta_i = 1}\gamma_{1i} \sum\limits_{j:\eta_j=0} a_{ij}^2\gamma_{2j} \right) | \gamma_2^{1-\eta}, \eta \right]\\
	\le  \mathbb{E} \left[\exp\left( \sum\limits_{i: \eta_i = 1} 1.2 \pi_{1i} \tau \sum\limits_{j:\eta_j=0} a_{ij}^2\gamma_{2j}  \right) |  \eta \right]\\
	= \prod\limits_{j:\eta_j=0} \mathbb{E} \left[\exp\left(  1.2 \tau \gamma_{2j} \sum\limits_{i: \eta_i = 1} \pi_{1i} a_{ij}^2  \right) |  \eta \right]\\
	\le \prod\limits_{j:\eta_j=0} \exp\left(  1.2^2 \tau \pi_{2j} \sum\limits_{i: \eta_i = 1} \pi_{1i} a_{ij}^2  \right).
	\end{array}
	$$
	We apply the similar argument used for $\gamma_{2j}, j \in \{j: \eta_j=0\}$ to $\gamma_{1i}, i \in \{i: \eta_j=1\}$ in the last inequality. Note that $\tau \le 1/(4||A||_2^2)$ is also sufficient since
	$$
	1.2\tau \sum\limits_{i: \eta_j=1} a_{ij}^2\pi_{1i} \le 1.2\tau \max_{j} \sum_i a_{ij}^2 \le 1.2\tau ||A||_2^2 \le 0.3.
	$$
	Finally, we observe that
	$$
	\exp\left(  1.2^2 \tau \sum\limits_{j:\eta_j=0}  \sum\limits_{i: \eta_i = 1}  a_{ij}^2 \pi_{1i} \pi_{2j}\right)  \le 
	\exp\left(  1.2^2 \tau \sum\limits_{i \neq j}  a_{ij}^2 \pi_{1i} \pi_{2j}\right),
	$$
	which concludes the proof.
\end{proof}

\subsubsection*{Step 5. Combining things together}
Assume $|\lambda| < 1/(\sqrt{8C_3} || A||_2)$. Combining all the steps above, we get
$$
\begin{array}{l}
\mathbb{E} \exp(\lambda S_{\rm off})\\
\le 
\mathbb{E}_\eta \mathbb{E}_\gamma  \mathbb{E}_X \exp(4\lambda S_\eta)\\
\le \mathbb{E}_\eta \mathbb{E}_\gamma  \mathbb{E}_{Z_1^\eta} \mathbb{E}\left(
\exp(4\lambda S_\eta) | \gamma_1, \gamma_2, \eta, Z_1^\eta
\right) \\
\le \mathbb{E}_\eta  \tilde{f}_{\eta}\\
\le  \exp\left(
2.88 C_3 \lambda^2 \sum\limits_{i \neq j} a_{ij}^2 \pi_{1i} \pi_{2i}
\right),
\end{array}
$$
which proves Lemma \ref{lem:cross_product_mgf}.

\section{Proof of Theorem 2}\label{app:proof_HWineq_nzm}
Let us define the centered sub-Gaussian variable $Z_{ij} = \tilde{Z}_{ij} - \mu_{ij}$. Then, it is easy to check that the bilinear form is decomposed into 8 terms.
$$
\begin{array}{rcl}
S &\equiv&  (\tilde{\bZ}_1 * \bgamma_1)^\top \bA (\tilde{\bZ}_2 * \bgamma_2) - \mathbb{E}(\tilde{\bZ}_1 * \bgamma_1)^\top \bA (\tilde{\bZ}_2 * \bgamma_2) \\
& =& \sum\limits_{i,j}a_{ij} \gamma_{1i}\gamma_{2j} Z_{1i}Z_{2j} - \mathbb{E}\left[\sum\limits_{i,j}a_{ij} \gamma_{1i}\gamma_{2j} Z_{1i}Z_{2j}\right] \\
&& + \sum\limits_{i,j}a_{ij} \mu_{1i} \gamma_{1i}\gamma_{2j}Z_{2j} \\
&& + \sum\limits_{i,j}a_{ij} \mu_{2j} \gamma_{1i}\gamma_{2j}Z_{1i} \\
&& + \sum\limits_{i,j}a_{ij} \mu_{1i}\mu_{2j} \gamma_{1i}\gamma_{2j} - \mathbb{E}\left[\sum\limits_{i,j}a_{ij} \mu_{1i}\mu_{2j} \gamma_{1i}\gamma_{2j}\right]\\
&=& \sum\limits_{i,j}a_{ij} \gamma_{1i}\gamma_{2j} Z_{1i}Z_{2j} - \mathbb{E}\left[\sum\limits_{i,j}a_{ij} \gamma_{1i}\gamma_{2j} Z_{1i}Z_{2j}\right] \\
&& + \sum\limits_{i,j}a_{ij} \mu_{1i} (\gamma_{1i}-\pi_{1i})\gamma_{2j}Z_{2j} + \sum\limits_{i,j}a_{ij} \mu_{1i} \pi_{1i}\gamma_{2j}Z_{2j}\\
&& + \sum\limits_{i,j}a_{ij} \mu_{2j} \gamma_{1i}Z_{1i} (\gamma_{2j}-\pi_{2j}) + \sum\limits_{i,j}a_{ij} \mu_{2j}\pi_{2j} \gamma_{1i}Z_{1i} \\
&& + \sum\limits_{i,j}a_{ij} \mu_{1i}\mu_{2j} (\gamma_{1i}-\pi_{1i})(\gamma_{2j} - \pi_{2j}) + 
\sum\limits_{i,j}a_{ij} \mu_{1i}\mu_{2j} \pi_{2j} (\gamma_{1i}-\pi_{1i}) + 
\sum\limits_{i,j}a_{ij} \mu_{1i} \pi_{1i} \mu_{2j}  (\gamma_{2j} - \pi_{2j})\\
&\equiv& S_{1,1} \\
&& + S_{2,1} + S_{2,2}\\
&& + S_{3,1} + S_{3,2}\\
&& + S_{4,1} + S_{4,2} + S_{4,3}\\
\end{array}
$$
We will use Theorem \ref{thm:HWineq_claim} for the bilinear forms $S_{1,1}, S_{2,1}, S_{3,1}, S_{4,1}$, and Hoeffding's inequality, stated below, for the linear combinations $S_{2,2}, S_{3,2}, S_{4,2}, S_{4,3}$. 
For $S_{2,1}=\sum\limits_{i,j}a_{ij} \mu_{1i} (\gamma_{1i}-\pi_{1i})\gamma_{2j}Z_{2j}$, we treat $\gamma_{1i}-\pi_{1i}$ as the centered sub-Gaussian variables with $\psi_2$-norm at most $1$ in which the $n$ sub-Gaussian are completely observed.
Applying the union bound and combining the results of $S_{i,j}$, we can derive the bound for ${\rm P} \Big[\big| S \big| > t \Big]$.
\begin{theorem}[Hoeffding's inequality]\label{thm:hoeffing}
	Let $Z_i$, $i=1,\ldots, n$, be mean-zero independent sub-Gaussian variables satisfying $\max\limits_{1 \le i \le n} ||Z_i||_{\psi_2} \le K$, and $\gamma_i$, $i=1,\ldots, n$, be independent Bernoulli variables. Also, suppose $Z_i$ and $\gamma_i$ are independent for all $i$. Let $\balpha = (\alpha_i, 1\le i \le n)^\top \in \mathbb{R}^n$ be a given coefficient vector. Then, we have
	$$
	\mathbb{E}\exp\left(
	\lambda \sum_{i=1}^n \alpha_i \gamma_i Z_i
	\right) \le \exp\left(
	C \lambda^2 K^2 \sum_{i=1}^n \alpha_i^2
	\right), \quad \lambda>0, 
	$$
	for some numerical constant $C>0$. Moreover, we have
	$$
	{\rm P} \bigg[
	\big| \sum\limits_{i=1}^n \alpha_i \gamma_i Z_i \big| > t 
	\bigg] \le 2\exp \left\{
	- \dfrac{ct^2}{K^2  ||\balpha||_2^2}
	\right\}, \quad t > 0, 
	$$
	for some numerical constant $c>0$.
\end{theorem}

\end{document}